\documentclass[12pt]{amsart}
\usepackage[left=3cm,top=3cm,right=3cm]{geometry}
\usepackage{bbm}
\usepackage{mathrsfs}
\usepackage[usenames,dvipsnames,svgnames,table]{xcolor}
\usepackage[pdftex,pagebackref,colorlinks=true,linkcolor=BrickRed,citecolor=OliveGreen,pdfstartview=FitH]{hyperref}
\usepackage{amssymb}

\usepackage{graphicx}
\usepackage{hyperref}
\usepackage{float}
\usepackage[normalem]{ulem}
\usepackage{verbatim}

\makeatletter
\newsavebox{\@brx}
\newcommand{\llangle}[1][]{\savebox{\@brx}{\(\m@th{#1\langle}\)}%
  \mathopen{\copy\@brx\kern-0.5\wd\@brx\usebox{\@brx}}}
\newcommand{\rrangle}[1][]{\savebox{\@brx}{\(\m@th{#1\rangle}\)}%
  \mathclose{\copy\@brx\kern-0.5\wd\@brx\usebox{\@brx}}}
\makeatother

\newtheorem{thm}{Theorem}[section]

\newtheorem{coro}[thm]{Corollary}
\newtheorem{lem}[thm]{Lemma}
\newtheorem{prop}[thm]{Proposition}
\newtheorem{conj}[thm]{Conjecture}

\theoremstyle{definition}

\newtheorem{defn}[thm]{Definition}

\newtheorem{nota}[thm]{Notation}

\theoremstyle{remark}

\newtheorem{remk}[thm]{Remark}

\renewcommand{\phi}{\varphi}

\newcommand{\Rbb}{ {\mathbb R}}

\newcommand{\Cbb}{ {\mathbb C}}

\newcommand{\Zcal}{ {\mathcal Z}}

\newcommand{\Lcal}{ {\mathcal L}}

\newcommand{\be}{\begin{enumerate}}
\newcommand{\ee}{\end{enumerate}}

\renewcommand{\paragraph}{\bigskip \noindent}

\title{The hot spots conjecture for some non-convex polygons}

\author{Lawford Hatcher}

\begin{document}

\begin{abstract}
    We give an elementary new proof of the hot spots conjecture for L-shaped domains. This result, in addition to a new eigenvalue inequality, allows us to locate the hot spots in Swiss cross translation surfaces. We then prove, in several cases, that first mixed Dirichlet-Neumann eigenfunctions of the Laplacian on L-shaped domains also have no interior critical points. As a combination of these results, we prove the hot spots conjecture for five classes of domains tiled by L-shaped domains, including a class of non-simply connected domains. An interesting feature of the proofs is that we make positive use of the \textit{lack} of regularity of eigenfunctions on non-convex polygons.
\end{abstract}

\maketitle

\section{Introduction}
Let $P\subseteq \Rbb^2$ be a bounded polygonal domain, and let $D,N\subseteq \partial P$ be relatively open, disjoint sets such that $\overline{D\cup N}=\partial P$. We consider in this paper the critical points of first non-constant eigenfunctions of the following mixed Dirichlet-Neumann eigenvalue problem:
\begin{equation}\label{mixedeqn}
    \begin{cases}
        -\Delta u=\lambda u\;\;&\text{in}\;\;P\\
        u\equiv 0\;\;&\text{in}\;\;D\\
        \partial_{\nu}u\equiv 0\;\;&\text{in}\;\;N,
    \end{cases}
\end{equation}
where $\partial_{\nu}$ denotes the outward pointing normal derivative. When $D=\emptyset$, this is known as the Neumann eigenvalue problem, and we denote the eigenvalues by $0=\mu_1<\mu_2\leq...$. If $D\neq \emptyset$, we denote the eigenvalues by $0<\lambda_1^D<\lambda_2^D\leq ...$  When we wish to emphasize that the eigenvalues depend on the domain itself, we will write $\mu_k(P)$ or $\lambda_k^D(P)$. In the following, we consider endpoints of $D$ and $N$ to be vertices of $P$, even if the angle at these endpoints equals $\pi$. Using the boundary conditions, we may analytically extend each eigenfunction to the interior of each edge of $P$. We never consider vertices of $P$ to be critical points of eigenfunctions, but this extension allows us to discuss critical points in the interior of each edge. \\
\indent The hot spots conjecture of Rauch \cite{rauch} asserts that each second Neumann eigenfunction of the Laplacian on a bounded Lipschitz domain in $\Rbb^n$ attains its extrema only on the boundary of the domain.\footnote{Due to the recently announced counterexample of de Dios Pont \cite{diospont} as well as the older counterexamples of Burdzy-Werner \cite{burdzywerner}, the conjecture is now widely believed to hold for contractible domains in sufficiently low dimensions.} A natural extension of this problem is to ask for which triples $(P,D,N)$ does a first mixed eigenfunction have extrema only in $\partial P$. In this paper, we address each of these questions for various polygonal domains in $\Rbb^2$ whose edges are all either vertical or horizontal. We begin by giving a new proof of the hot spots conjecture for L-shaped domains. By studying other qualitative features of second Neumann eigenfunctions of L-shaped domains, we are able to determine the locations of critical points of first non-constant eigenfunctions of a certain class of closed surfaces with a flat structure known as Swiss cross translation surfaces (see Definition \ref{swissdef} below). We then extend our results to various mixed problems on L-shaped domains and rectangles. Combining several of these results, we resolve the hot spots conjecture for T-shaped, U-shaped, O-shaped, H-shaped, and Swiss cross domains, defined below. In the following subsections, we give more precise statements of these results.

\subsection{L-shaped domains}

\begin{defn}[L-shaped domains]\label{Ldef}
    Let $a_1,a_2,a_3,a_4>0$, and let $L\subseteq\Rbb^2$ be an open hexagon with vertices at $(0,0),(a_1+a_2,0),(a_1+a_2,a_3),(a_1,a_3),(a_1,a_3+a_4),$ and $(0,a_3+a_4)$. See Figure \ref{example}. We call such a domain (and domains isometric to such a domain) an \textit{L-shaped domain}. We say that an L-shaped domain embedded in the plane in this way is \textit{canonically embedded}. Because the diameter of $L$ equals the distance between 
$(a_1+a_2,0)$ and $(0,a_3+a_4)$, we call these vertices the \textit{diametric vertices} of $L$. We refer to the two edges of $L$ that are contained in the coordinate axes as the \textit{long outer edges} of $L$, the two other edges not adjacent to the non-convex vertex as the \textit{short outer edges}, and the two edges adjacent to the non-convex vertex as the \textit{inner edges} of $L$.
\end{defn}
\begin{figure}
    \centering
    \includegraphics[scale=0.2]{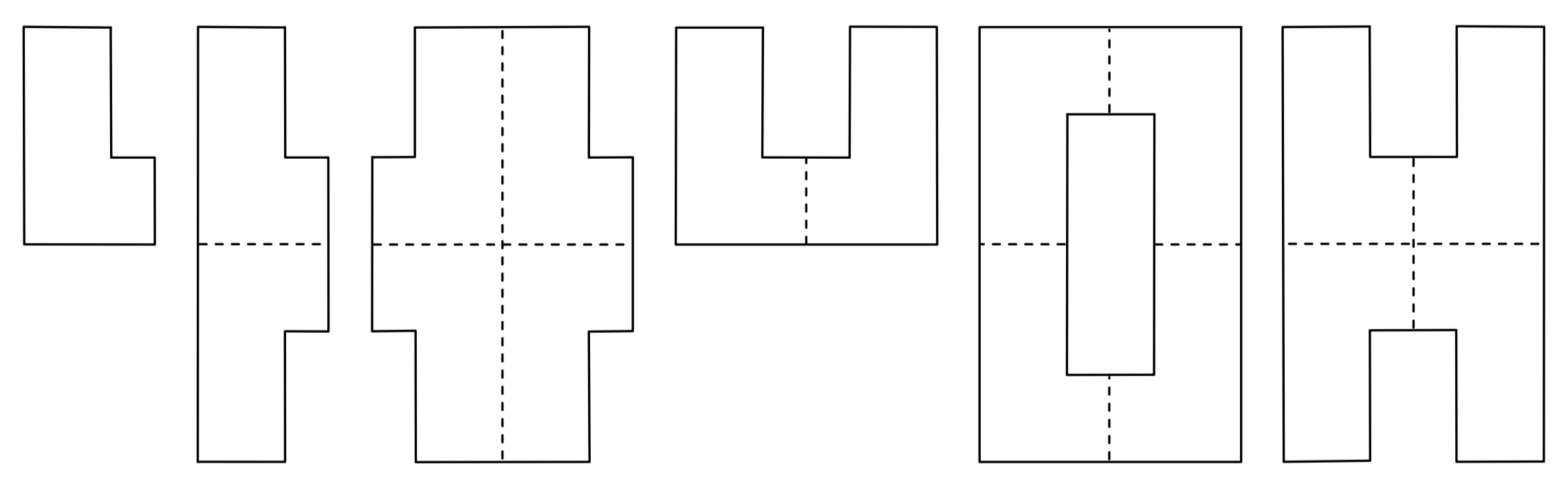}
    \caption{An example of an L-shaped domain and an example of each type of L-tiled domain. From left to right, the figure shows an L-shaped domain, a T-shaped domain, a Swiss cross domain, a U-shaped domain, an O-shaped domain, and an H-shaped domain. The dashed lines reveal the L-shaped domain by which each L-tiled domain is tiled.}
    \label{example}
\end{figure}
\begin{thm}\label{mainthm}
    Let $L$ be an L-shaped domain, and let $u$ be a second Neumann eigenfunction of $L$. Then $u$ has no (non-vertex) critical points. In particular, the extrema of $u$ occur only at the diametric vertices of $L$. Moreover, if $L$ is canonically embedded, then $u$ may be chosen such that $\partial_xu>0$ and $\partial_yu<0$ in $L$.
\end{thm}
Of course, L-shaped domains are examples of ``lip domains" in the sense of Atar and Burdzy, for which they proved the hot spots conjecture in their main theorem in \cite{lipdomains}. However, their proof involves complicated techniques using reflected Brownian motion. Rohleder recently published a new, non-probabilistic proof for lip domains that do not have non-convex corners \cite{newapproach} and thus does not apply to L-shaped domains. Our proof, though not immediately generalizable to all lip domains, is significantly simpler than the first of these proofs. \\
\indent Since Theorem \ref{mainthm} implies that second Neumann eigenfunctions on L-shaped domains cannot vanish at the diametric vertices, it follows that the second Neumann eigenspace is one dimensional (This is also a consequence of the main result of \cite{lipdomains}).
\begin{coro}\label{simple}
    The second Neumann eigenvalue for any L-shaped domain is simple. 
\end{coro}
Theorem \ref{mainthm} also provides some information on the location of the nodal line of second Neumann eigenfunctions on L-shaped domains:
\begin{coro}\label{nodalline}
    The nodal set (i.e., $u^{-1}(\{0\})$) of a second Neumann eigenfunction on an L-shaped domain $L$ is a simple arc with one endpoint in the closure of a long outer edge and the other endpoint in the closure of either a short outer edge or an inner edge. In particular, the endpoints of the nodal set cannot both be in short outer edges and cannot both be in long outer edges.\\
    \indent If $a_1=a_3$ and $a_2=a_4$ (i.e. $L$ is symmetric about the line $y=x$), then the nodal set is the straight line segment with endpoints at $(0,0)$ and $(a_1,a_1)$.
\end{coro}

\subsection{Swiss cross translation surfaces}
Beyond the hot spots conjecture, it is an interesting question to ask where the hot spots on closed surfaces occur. We will later prove (see Theorem \ref{strictineq}) an eigenvalue inequality that strengthens a result of Lotoreichik and Rohleder \cite{lotoroh} and which may be of independent interest. Using this inequality and our knowledge of the location of the hot spots on L-shaped domains, we explicitly locate and classify the critical points (and in particular, the extrema) of each first non-constant eigenfunction of the Laplacian on a class of genus $2$ translation surfaces in the stratum $\mathcal{H}(2)$.

\begin{nota}\label{edgelabels}
    For a canonically embedded L-shaped domain $L$, we will denote the interior of the edge of $L$ contained in the $x$-axis by $e_1$. The interiors of the other edges will be labelled in counterclockwise order (e.g., the edge contained in the $y$-axis equals $e_6$).
\end{nota}

\begin{defn}[T-shaped domains, Swiss cross domains, and Swiss cross surfaces]\label{swissdef}
    Let $L$ be a canonically embedded L-shaped domain. Let $L_1$ be an isometric copy of $L$ obtained by reflection over $e_1$. Let $T=L\cup L_1\cup e_1$. We will call $T$ a \textit{T-shaped domain}. Let $T_1$ be an isometric copy of $T$ obtained by reflection over the edge $e$ containing $e_6$. Let $C=T\cup T_1\cup e$. We will call $C$ a \textit{Swiss cross domain}. See Figure \ref{example}. Let $S$ be the translation surface $\overline{C}/\sim$ where $\overline{C}$ is the closure of $C$ in $\Rbb^2$ and where $\sim$ identifies opposite pairs of edges by horizontal or vertical translation. We will call $S$ a \textit{Swiss cross surface}. The non-convex vertices of $C$ descend to a single point in $S$, and we refer to the image of these points as the \textit{cone point} of $S$. In what follows we will identify $L$ with its images in $T$, $C$, and $S$.
\end{defn}

Using the natural Euclidean metric on $S$ minus its cone point, we may define a Friedrichs Laplacian, which admits an orthonormal basis of $L^2(S)$-eigenfunctions and eigenvalues (see the paper of Hillairet \cite{luc}). The lowest eigenvalue is $0$ whose corresponding eigenfunction is constant. Let $u$ denote a first non-constant eigenfunction for $S$. Using the Euclidean gradient away from the cone point, we use the standard definition of critical points. Let $p$ denote the cone point of $S$. If there is a neighborhood $U$ of $p$ such that $U\cap \big(u^{-1}(\{u(p)\})\setminus \{p\}\big)$ contains a number of disjoint arcs not equal to $2$, then we consider $p$ to be a critical point of $u$. 

\begin{thm}\label{swissthm}
    Let $S$ be a Swiss cross surface tiled by an L-shaped domain $L$. The first positive eigenvalue of $S$ is simple and equals the second Neumann eigenvalue of $L$. Each eigenfunction $u$ corresponding to this eigenvalue has exactly $6$ critical points, of which exactly $2$ are the extrema of $u$. Each critical point is located at a vertex of $L$, and the extrema occur where the second Neumann eigenfunction of $L$ attains its extrema. In particular, the non-diametric vertices of $L$ are saddle critical points of $u$.
\end{thm}

\subsection{Hot spots of first mixed Dirichlet-Neumann eigenfunctions}
\indent We next consider the set of critical points for first mixed eigenfunctions (that is, eigenfunctions corresponding to $\lambda_1^D$ when $D\neq\emptyset$), whose study has gained attention in recent months (see \cite{me}, \cite{liyao}, and \cite{rohledermixed}). The paper \cite{rohledermixed} gives a particularly nice overview of older results in this area. Though it is a difficult problem to determine for which triples $(P,D,N)$ the first mixed eigenfunction has no interior critical points, we make the following conjecture:

\begin{conj}\label{mixedconj}
    Suppose that $\Omega$ is a simply connected Lipschitz domain and that $D$ equals a line segment. Then each first mixed eigenfunction has no interior extrema. 
\end{conj}

The next theorem proves that Conjecture \ref{mixedconj} holds in the case of rectangles. When $D$ equals an entire edge of a rectangle, this result is a simple computation, but when $D$ is strictly contained in an edge of the rectangle, the eigenfunctions cannot be explicitly computed, and the result is non-trivial. We also prove Conjecture \ref{mixedconj} in the case of L-shaped domains when $D$ equals an edge of $L$.

\begin{thm}\label{rectangleconj}
    Let $R$ be a rectangle, and let $D\subseteq \partial R$ be a line segment contained in an edge $e$ of $R$. Then the critical points of each first mixed eigenfunction are contained in the edge opposite to $e$. In particular, Conjecture \ref{mixedconj} holds for rectangles. 
\end{thm}

\begin{thm}\label{mixedthm}
    Let $L$ be a canonically embedded L-shaped domain, and suppose that the pair $(L,D)$ satisfies one of the following conditions:
    \begin{enumerate}
        \item $D$ is a single edge or is any union of parallel edges. Additionally, if $D=e_1\cup e_5$ (resp. $D=e_2\cup e_6$), then $a_3\neq a_4$ (resp. $a_1\neq a_2$),
        \item $D$ is the union of the inner edges of $L$ with an outer edge, or
        \item $D$ is the union of the inner edges with two parallel outer edges.
    \end{enumerate}
    Then each first mixed eigenfunction has no hot spots. Moreover, there exists a direction in which the eigenfunction is strictly monotonic in (the interior of) $L$. 
\end{thm}
The second hypothesis in (1) is strictly necessary. Indeed, if $D=e_1\cup e_5$ and $a_3=a_4$, then one can check that each first mixed eigenfunction is a scalar multiple of $(x,y)\mapsto \sin(\frac{\pi}{2a_3}y)$, whose critical set (in fact, its set of global extrema) contains a line segment intersecting $L$. It is interesting to observe that arbitrarily small perturbations of these parameters yield first mixed eigenfunctions with no interior critical points.

\subsection{Hot spots in L-tiled domains}
\begin{defn}[U-shaped domains, O-shaped domains, and H-shaped domains]
    Let $L$ be a canonically embedded L-shaped domain. Let $L_1$ be an isometric copy of $L$ obtained by reflection over $e_2$. Let $U=L\cup L_1\cup e_2$. We call $U$ a \textit{U-shaped domain}. Let $e$ be the union of the two isometric copies of $e_5$ contained in $\partial U$. Let $U_1$ be an isometric copy of $U$ obtained by reflection $e$. Let $O=U\cup U_1\cup e$. We call $O$ an \textit{O-shaped domain}. Let $e'$ be the edge of $U$ that contains $e_1$. Let $U_2$ be an isometric copy of $U$ obtained by reflection over $e'$. We call $H=U\cup U_2\cup e'$ an \textit{H-shaped domain}. We say a domain is \textit{L-tiled} if it is a T-shaped, U-shaped, O-shaped, H-shaped, or Swiss cross domain. See Figure \ref{example}.
\end{defn}

In two special cases of the hypotheses in Theorem \ref{mixedthm}, we show in Theorems \ref{longdirichlet} and \ref{shortouterdirichlet} below that each first mixed eigenfunction has exactly one local extremum, and it occurs at the diametric vertex farthest from $D$. Using this fact, we can prove the hot spots conjecture for all L-tiled domains. The case of O-shaped domains is especially interesting because these domains are doubly connected, and the hot spots conjecture is known to be false in general in this setting (see \cite{burdzywerner}, \cite{burdzy}, and \cite{kleefeld}). 

\begin{thm}\label{Ltiled}
    Second Neumann eigenfunctions of L-tiled domains have no interior critical points. In particular, the hot spots conjecture holds for these domains. Moreover, every second Neumann eigenfunction of an L-tiled domain is monotonic in the direction of at least one of the edges of the domain. 
\end{thm}

By specifying the parameters for an L-shaped domain, the spaces of L-, T-, U-, O-, and H-shaped domains and Swiss cross domains can be identified with the positive cone in $\Rbb^4$. By calling a property of one of these families \textit{generic}, we mean that the property holds outside of a set of Lebesgue measure zero in this cone. 

\begin{thm}\label{genericsimplicity}
    The second Neumann eigenvalue for L-tiled domains is generically simple and always of multiplicity at most two. Every class of L-tiled domain contains a domain with a two-dimensional second Neumann eigenspace.  
\end{thm}

\begin{remk}
    In the case of Swiss cross domains with simple second Neumann eigenvalue, one may combine Lemma \ref{atleastone} below with Theorem 1.1 of Jerison and Nadirashvili's paper \cite{jerisonnad} to see that the hot spots conjecture holds. However, the result does not follow from Theorem 1.1 of \cite{jerisonnad} in the case where the second eigenvalue is multiple or for other L-tiled domains. Furthermore, even though Swiss cross, O-shaped, and H-shaped domains have two orthogonal axes of symmetry, Theorem 1.4 of \cite{jerisonnad} does not apply to these domains since they are not convex. 
\end{remk}

\section{Some preparatory results}\label{technical}
We establish in this section several technical analytical results that will reduce proofs of several of our theorems to arguments regarding the behavior of zero-level sets of various derivatives of eigenfunctions.\\
\indent Let $P\subseteq\Rbb^2$ denote a bounded polygonal domain (i.e. a connected open set) with vertex set $V$. Let $\phi$ be a Laplace eigenfunction of $P$\footnote{We do not suppose that $\phi$ satisfies any specific boundary conditions.} that extends analytically to $\partial P\setminus V$ and continuously to $V$. In what follows, we will identify $\phi$ with its extension. We will denote the zero-level set of $\phi$ by $\Zcal(\phi)$, and we will call this set the \textit{nodal set} of $\phi$. We will call a connected component of $P\setminus \Zcal(\phi)$ a \textit{nodal domain} for $\phi$. We begin by recalling Proposition 3.2 of the paper of Judge-Mondal \cite{judgemondal}\footnote{The paper \cite{judgemondal} has a significant erratum \cite{erratum}. We make use only of results from \cite{judgemondal} that were unaffected by the error in Lemma 3.4 \cite{judgemondal}.}:

\begin{lem}\label{nicenodalsets}
    The nodal set of $\phi$ is a union of properly immersed\footnote{By \textit{properly immersed}, we mean that the arcs may be parameterized such that the pre-image of any compact set is compact. In particular, these arcs have endpoints only at $\partial P$.} analytic open arcs and isolated points in $\partial P$. In particular, $\phi$ has no isolated zeros in $P$.
\end{lem}

We will implicitly use this result throughout the paper to say that whenever some directional derivative of a mixed or Neumann eigenfunction\footnote{Note that a directional derivative of a Laplace eigenfunction is also a Laplace eigenfunction.} vanishes in $P$, then the nodal set of this derivative actually contains an arc that intersects $P$. The following lemma shows, for example, that in the case of a first mixed or second Neumann eigenfunction, these nodal sets never contains a loop unless this loop intersects a non-convex vertex of $P$. Let $D,N\subseteq \partial P$ be as in the introduction, and let $\lambda_1^D$ be the corresponding first mixed eigenvalue for the problem (\ref{mixedeqn}). If $D=\emptyset$, then let $\mu_2$ be the second Neumann eigenvalue.

\begin{lem}\label{noloops}
    Suppose that $U\subseteq P$ is an open set such that $P\setminus U$ has non-empty interior. Suppose that $\phi\in H^1(U)\setminus\{0\}$ is smooth in $U$ and satisfies $-\Delta u=\lambda u$ for some $\lambda$. Further suppose that for each $x\in\partial U$, if $\phi(x)\neq 0$, then $(\partial_{\nu}\phi)(x)=0$ and $x\in N$. If $D\neq \emptyset$, then $\lambda_1^D<\lambda$. If $D=\emptyset$ and the set $\{x\in N:\phi(x)\neq 0\}$ is contained in a finite union of parallel lines, then $\mu_2<\lambda$.
\end{lem}
\begin{proof}
    First suppose that $D\neq \emptyset$. Since $\phi$ vanishes on $\partial U\cap P$, we may extend $\phi$ by $0$ to be an element of $H^1(P)$, which we will also call $\phi$. Then $\phi$ vanishes on $D$ and satisfies either Dirichlet or Neumann conditions at each point of $N$. Because $\phi$ is smooth in $U$ and satisfies the eigenfunction equation, we may integrate by parts to get $$\int_P|\nabla \phi|^2=\lambda\int_P|\phi|^2+\int_{\partial P}\phi\partial_{\nu}\overline{\phi}=\lambda\int_P|\phi|^2.$$ By the variational formula for $\lambda_1^D$, we see that $\lambda_1^D\leq \lambda$, and if equality holds, then $\phi$ is actually equal to the first mixed eigenfunction. Since $\phi$ vanishes on an open set and is non-constant, this contradicts unique continuation. Thus, $\lambda_1^D<\lambda$.\\
    \indent Now suppose that $D=\emptyset$ and that $S:=\{x\in N:\phi(x)\neq 0\}$ is contained in a finite union of parallel lines. By the same reasoning of the above paragraph, we see that $\lambda$ is strictly greater than the first mixed eigenvalue corresponding to the problem with Neumann boundary conditions on $S$ and Dirichlet boundary conditions on $\partial P\setminus S$. The inequality then follows from Theorem \ref{strictineq}\footnote{Theorem 3.1 of \cite{lotoroh} is also sufficient to establish this inequality.} below (this is where the assumption that $S$ is contained in a finite union of parallel lines is used).
\end{proof}

The next lemma gives partial information on the location of the critical points for first mixed eigenfunctions on polygons. Let $u$ denote a first mixed eigenfunction for some non-empty $D,N\subseteq \partial P$.
\begin{lem}\label{nodircp}
    $u$ has no critical points in $D$. 
\end{lem}
\begin{proof}
    Recall that first mixed eigenfunctions are either positive or negative in $P$. The Hopf lemma \cite{hopf} therefore implies the lemma. 
\end{proof}

\begin{lem}\label{cpinedge}
    Suppose that $u$ is a Laplace eigenfunction on a polygon $P$ and that $u$ satisfies Neumann conditions on a horizontal edge $e$. If $\partial_xu\not\equiv 0$ on $e$, then each critical point of $u$ in $e$ is an endpoint of at least one arc in $\Zcal(\partial_xu)$ that intersects $P$.   
\end{lem}
\begin{proof}
    Let $p\in e$ be a critical point of $u$. Extend $u$ to a neighborhood of $p$ by reflection. By Lemma \ref{nicenodalsets}, there exists an arc in $\Zcal(\partial_xu)$ passing through $p$. Since this arc is not contained in $e$ and using the evenness of $\partial_xu$, this arc must intersect $P$. 
\end{proof}
\begin{remk}
    For a simply connected polygon $P$ with connected Dirichlet region $D$, we know exactly when a second Neumann or first mixed eigenfunction $u$ can satisfy $\partial_xu\equiv 0$ on $e$. Namely, this occurs if and only if every point on $e$ is a critical point of $u$. By Theorem 1 of Judge and Mondal's paper \cite{remarksoncritset}, if $P$ is simply connected and not a rectangle, then a second Neumann eigenfunction has at most finitely many critical points, so $\partial_xu\not\equiv 0$ on $e$. By Theorem 1.2 of the author's paper \cite{me}, if $D$ is connected and the pair $(P,D)$ is not a rectangle with $D$ equal to an edge of $P$, then each first mixed eigenfunction has at most finitely many critical points, so $\partial_xu\not\equiv 0$ on $e$.
\end{remk}

For the following two results, we use the Bessel expansion employed in \cite{judgemondal} in a neighborhood of the non-convex vertex. Let $P$ be polygon with a vertex $v$ with angle $3\pi/2$ located at the origin with $P$ rotated such that its intersection with $\{(x,y):x>0,y<0\}$ is empty. Let $u$ be a Laplace eigenfunction for $P$ satisfying Neumann boundary conditions near $v$. Then in a neighborhood of $v$, using polar coordinates, $u$ has a series expansion of the form 
\begin{equation}\label{neumannexpansion}
    u(r,\theta)=\sum_{n=0}^{\infty}c_nr^{\frac{2}{3}n}g_n(r^2)\cos\Big(\frac{2}{3}n\theta\Big),
\end{equation}
where $c_n\in \Rbb$ and the $g_n$ are analytic functions with $g_n(0)=1$ and $g_n'(0)\neq0$. When discussing a particular vertex $v$, eigenfunction $u$, and coefficient $c_n$ in this expansion, we will write $c_n(u,v)$. \\
\indent For an open set $\Omega\subseteq \Rbb^2$ and $k\geq 0$, let $H^k(\Omega)$ denote the Sobolev space of $L^2(\Omega)$-functions whose (weak) derivatives of order less than or equal to $k$ exist and are square-integrable. Expansion (\ref{neumannexpansion}) is useful for studying the regularity of Laplace eigenfunctions.

\begin{lem}\label{regularity}
    Let $P$ be a polygon whose edges are all either vertical or horizontal. Let $u$ be a Laplace eigenfunction for $P$ satisfying Neumann conditions near the non-convex vertices and satisfying Dirichlet or Neumann conditions elsewhere on the boundary. Let $V$ be any neighborhood of the non-convex vertices. Then $u\in H^k(P\setminus \overline{V})$ for all $k\geq 0$. Moreover, $u\in H^2(P)$ if and only if $c_1(u,v)=0$ for every non-convex vertex $v$ of $P$.
\end{lem}
\begin{proof}
    Using the Neumann (resp. Dirichlet) boundary condition, we may extend $u$ via reflection to be an even (resp. odd) real analytic function in a neighborhood of the interior of each Neumann (resp. Dirichlet) edge of $P$. By reflecting twice, we may also extend $u$ to be analytic in a neighborhood of each right-angled vertex (i.e. each vertex with interior angle equal to $\pi/2$). This proves the first statement.\\
    \indent The second statement follows by combining the first statement with a direct computation of integrals of derivatives of the expansion (\ref{neumannexpansion}).
\end{proof}

Let $P$ be a polygon whose edges are all horizontal or vertical and that has a non-convex vertex $v$. Suppose that $P$ is embedded in $\Rbb^2$ with $v$ at the origin as in expansion (\ref{neumannexpansion}). Let $u$ be a Laplace eigenfunction satisfying Neumann conditions near $v$.

\begin{lem}\label{neumannarc}
    If $c_1(u,v)\neq 0$, then $v$ is the endpoint of exactly one arc in $\Zcal(\partial_xu)$, and, near the non-convex vertex, this arc is contained in the negative $y$-axis.
\end{lem}
\begin{proof}
    In polar coordinates, we have $\partial_x=\cos(\theta)\partial_r-\frac{1}{r}\sin(\theta) \partial_{\theta}$. Expansion (\ref{neumannexpansion}) gives $$\partial_xu(r,\theta)=\frac{2}{3}c_1r^{-\frac{1}{3}}\Big(\cos(\theta)\cos(\frac{2}{3}\theta)+\sin(\theta)\sin(\frac{2}{3}\theta)\Big)+O(r^{\frac{1}{3}})=\frac{2}{3}c_1r^{-\frac{1}{3}}\cos(-\frac{1}{3}\theta)+O(r^{\frac{1}{3}}).$$ For points away from the negative $y$-axis, since $c_1\neq 0$, the derivative $\partial_xu$ does not vanish. By the Neumann boundary condition, $\partial_xu$ vanishes on the negative $y$-axis near $v$. 
\end{proof}

Similarly to the expansion in the case of Neumann conditions near vertex, we have an expansion for the case when both edges of $P$ near the non-convex vertex $v$ are contained in $D$:
\begin{equation}\label{dirichletexpansion}
    u(r,\theta)=\sum_{n=1}^{\infty}d_nr^{\frac{2}{3}n}g_n(r^2)\sin\Big(\frac{2}{3}n\theta\Big).
\end{equation}
We use this expansion to prove a Dirichlet analogue of Lemma \ref{neumannarc}, letting $u$ be a first mixed eigenfunction of $P$.
\begin{lem}\label{dirichletarc}
    The non-convex vertex $v$ is the endpoint of exactly one arc in $\Zcal(\partial_xu)$, and, near the non-convex vertex, this arc is contained in the positive $x$-axis.
\end{lem}
\begin{proof}
    Since $u$ does not vanish in $P$, we have $d_1\neq 0$ in expansion (\ref{dirichletexpansion}). We then have $$\partial_xu(r,\theta)=\frac{2}{3}d_1r^{-1/3}\sin(\frac{-1}{3}\theta)+O(r^{1/3}),$$ and the result follows from a similar argument to that used in Lemma \ref{neumannarc}.
\end{proof}

Finally, we have an expansion for a mixed non-convex vertex $v$. Suppose that, near $v$, the positive $x$-axis is contained in $D$ and that the negative $y$-axis is contained in $N$. Then, near the non-convex vertex,
\begin{equation}\label{mixedexpansion}              u(r,\theta)=\sum_{n=0}^{\infty}a_nr^{\frac{2}{3}(n+\frac{1}{2})}g_n(r^2)\sin\Big(\frac{2}{3}(n+\frac{1}{2})\theta\Big),
\end{equation} and we have 

\begin{lem}\label{mixedarc}
    If $u$ is a first mixed eigenfunction, then the non-convex vertex $v$ is an endpoint of exactly two arcs in $\Zcal(\partial_xu)\setminus\{v\}$. Near $v$, these arcs are contained in $\partial L$. Moreover, $v$ is an endpoint of exactly one arc in $\Zcal(\partial_yu)$, and, near $v$, this arc is contained in (the interior of) $P$.
\end{lem}
\begin{proof}
    Since $u$ does not vanish in $P$, we have $a_0\neq 0$ in expansion (\ref{mixedexpansion}). We thus have $$\partial_xu(r,\theta)=\frac{1}{3}a_0r^{-\frac{2}{3}}\sin(-\frac{2}{3}\theta)+O(1)$$ and $$\partial_yu(r,\theta)=\frac{1}{3}a_0r^{-\frac{2}{3}}\cos(-\frac{2}{3}\theta)+O(1).$$ Both statements follow from these computations.
\end{proof}
We also consider nodal sets of derivatives near flat vertices.
\begin{lem}\label{flatmixedarc}
    Suppose that $D$ and $N$ meet at a point $v$ with angle $\pi$ such that, near $v$, $D$ and $N$ are contained in the $x$-axis. Let $u$ be a first mixed eigenfunction. Then $v$ is an endpoint of exactly one arc in $\Zcal(\partial_yu)$, and this arc is contained in $N$. Moreover, $u\in H^1(P)\setminus H^2(P)$. 
\end{lem}
\begin{proof}
    Suppose that $v$ equals the origin and that $D$ is contained in the positive $x$-axis. Then, near $v$, $u$ has the following expansion:
    $$u(r,\theta)=\sum_{n=0}^{\infty}d_nr^{n+\frac{1}{2}}g_{n+1/2}(r^2)\sin\Big(\big(n+1/2)\theta\Big).$$
    The first statement follows from a similar computation to what was performed in the proof of Lemma \ref{mixedarc}. The second follows from direct analysis of the expansion as in the proof of Lemma \ref{regularity}.
\end{proof}

\begin{lem}\label{rightangledmin}
    Let $v$ be a right-angled vertex of a polygonal domain $P$. If $u$ is a first mixed eigenfunction for $P$ satisfying Neumann boundary conditions near $v$ and $u(v)>0$, then $v$ cannot be a local minimum of $u$. 
\end{lem}
\begin{proof}
    Since $u$ does not change signs and is positive at $v$, we have $u>0$ on $\overline{L}\setminus D$. By extending $u$ via reflection near $v$, $u$ extends to be a non-constant superharmonic function near $v$ and thus cannot be a local minimum of $u$. 
\end{proof}

We next recall a result of Lotoreichik and Rohleder comparing first mixed eigenvalues for different choices of pairs $(D,N)$ (see Proposition 2.3 of \cite{lotoroh}):
\begin{lem}\label{Dinclusion}
    If $D\subseteq D'$ (so $N\supseteq N'$) such that $D'\setminus D$ has a non-empty (relative) interior, then $\lambda_1^D<\lambda_1^{D'}$.
\end{lem}

We end the section by observing that certain derivatives of mixed and Neumann eigenfunctions on certain polygonal domains satisfy useful boundary conditions. Note however that, by Lemma \ref{regularity}, these derivatives do not necessarily lie in the correct Sobolev spaces to be considered $L^2(P)$-eigenfunctions. Let $P$ be a polygon in $\Rbb^2$ embedded such that each edge of $P$ is either horizontal or vertical.

\begin{lem}\label{derivativeBCs}
    Let $D_v$ and $N_v$ (resp. $D_h$ and $N_h$) be the unions of vertical (resp. horizontal) line segments contained in $D$ and $N$, respectively. Let $u$ be a mixed eigenfunction for $P$, $D$, and $N$. Then $\partial_xu$ is an eigenfunction with the same eigenvalue as $u$. Furthermore, $\partial_xu$ satisfies Dirichlet boundary conditions on $N_v$ and $D_h$ and Neumann boundary conditions on $D_v$ and $N_h$. 
\end{lem}
\begin{proof}
    The first statement holds because $\partial_x$ commutes with the Laplacian: 
    $$-\Delta u=\lambda u\implies -\Delta(\partial_xu)=\partial_x(-\Delta u)=\lambda\partial_xu.$$
    That $\partial_xu$ satisfies Dirichlet conditions on $N_v$ and $D_h$ follows directly from the definitions. On $N_h$, this fact follows from the commutativity of partial derivatives:
    $$(\partial_{\nu}\partial_xu)|_{N_h}=\pm (\partial_y\partial_xu)|_{N_h}=\pm (\partial_x(\partial_yu))|_{N_h}=0.$$ On $D_v$, we use the fact that $u$ is an eigenfunction:
    $$(\partial_{\nu}\partial_xu)|_{D_v}=\pm \partial_x^2u|_{D_v}=\pm(-\lambda u-\partial_y^2u)|_{D_v}=0.$$
\end{proof}

\section{An eigenvalue inequality}\label{evalineq}
We prove in this section a strict version of the inequality given in Theorem 3.1 of \cite{lotoroh} for small eigenvalues on most polygons. We will use this eigenvalue inequality to locate the hot spots in L-shaped domains and Swiss cross surfaces. Let $P\subseteq \Rbb^2$ be a polygon. Let $D,N\subseteq \partial P$ be as in the introduction. Let $\lambda_1^D$ and $u$ be the first eigenvalue and a corresponding eigenfunction, respectively, of the mixed problem (\ref{mixedeqn}).\\
\indent The proof of Theorem \ref{strictineq} below is divided into two main cases: one that does not satisfy the below definition and one for which the first mixed eigenfunction is explicitly computable. 
\begin{defn}\label{badpair}
Let $P$ be a polygon, and suppose that there exists $a>0$ such that $P$, potentially after applying an isometry, satisfies
\begin{enumerate}
    \item Each edge of $P$ is either vertical or horizontal,
    \item $P$ is contained in the horizontal strip $\{(x,y)\mid 0<y<a\}$,
    \item $\partial P$ intersects $\{y=0\}$, and
    \item Each horizontal edge of $P$ is contained in one of three lines: $\{y=0\}$, $\{y=a/2\}$, or $\{y=a\}$.
\end{enumerate}
If $D$ is the intersection of $\partial P$ with $\{y=0\}\cup\{y=a\}$, then we say that the pair $(P,D)$ is a \textit{bad pair}.
\end{defn}

In Lemma \ref{nonconstant} below, $\partial_{\nu}$ is the outward normal derivative to $\Omega$.

\begin{lem}\label{nonconstant}
    If $P$ is any polygon and $D,N\subseteq \partial P$ are chosen such that $(P,D)$ is not a bad pair, then the restriction of $\partial_{\nu}u$ to each edge $e\subseteq D$ is non-constant. 
\end{lem}
\begin{proof}
    By applying an isometry if necessary, we may suppose that $e$ is contained in the $x$-axis and that, near $e$, $P$ lies in the upper half-plane. Let $U\subseteq\Rbb^2$ be an open neighborhood of $e$ such that each point of $U$ or its reflection over $e$ is contained in $\overline{\Omega}$. Using the Dirichlet condition on $e$ we can analytically extend $u$ to an odd eigenfunction about $e$ on $U$. We identify $u$ with its extension. Because of how $P$ is embedded, the directional derivative $-\partial_y$ extends the outward normal vector field to $e$. Let $v=-\partial_yu$, which is a Laplace eigenfunction with eigenvalue $\lambda_1^D$. Then $\partial_xv$ is also a Laplace eigenfunction with eigenvalue $\lambda_1^D$. Suppose that $v$ is constant on $e$, so $\partial_xv$ vanishes identically on $e$. By Lemma \ref{derivativeBCs}, $v$ satisfies Neumann conditions on $e$, so $\partial_xv$ also satisfies Neumann conditions on $e$. Hence, each point on $e$ is a critical point of $\partial_xv$. It then follows from Lemma 2 of \cite{remarksoncritset} that $\partial_xv$ is identically equal to $0$ on $U$, so $v$ is a function only of $y$. From this it follows that $v$ is a scalar multiple of the function $(x,y)\mapsto \cos(\sqrt{\lambda_1^D}y)$, so $u$ is a scalar multiple of $(x,y)\mapsto\sin(\sqrt{\lambda_1^D}y)$. Since we can scale $u$ such that $u>0$ in $P$, we have $$P\subseteq \Big\{(x,y)\mid 0<y<\frac{\pi}{\sqrt{\lambda_1^D}}\Big\},$$ and $D$ is contained in the boundary of this strip. Moreover, the only non-vertical edges of $P$ on which this function can satisfy Neumann conditions are contained in the horizontal line $\Big\{y=\frac{\pi}{2\sqrt{\lambda_1^D}}\Big\}$. Thus, $(P,D)$ must be a bad pair.
\end{proof}

Combining Lotoreichik and Rohleder's proof with Lemma \ref{nonconstant} yields a strict eigenvalue inequality:

\begin{thm}\label{strictineq}
    Suppose that $(P,D)$ does not consist of a rectangle $P$ with $D$ equal to two opposite shorter edges of $P$. Let $\mu_2$ be the first non-zero Neumann eigenvalue of the Laplacian on $P$. If, after applying an isometry if necessary, $N$ is a finite union of vertical line segments and $D$ contains a horizontal line segment, then $\mu_2<\lambda_1^D$.
\end{thm}
\begin{proof}
    Suppose to the contrary that $\lambda_1^D\leq \mu_2$. Let $u$ be a first mixed eigenfunction, and let $w(x,y)=e^{i\sqrt{\lambda_1^D}y}$. Note that $w$ satisfies Neumann boundary conditions on $N$ and $|\nabla w|^2=\lambda_1|w|^2$. For each $a,b\in\Cbb$, we get the following by integration by parts:
    \begin{align}
        \int_{\Omega}|\nabla(au+bw)|^2&=\lambda_1^D |a|^2\int_{\Omega}|u|^2+2\text{Re}\Big(a\overline{b}\int_{\Omega}\nabla u\cdot\nabla\overline{w}\Big)+\lambda_1^D|b|^2\int_{\Omega}|w|^2\\
        &=\lambda_1^D\Big(|a|^2\int_{\Omega}|u|^2+2\text{Re}\big(a\overline{b}\int_{\Omega}u\overline{w}\big)+|b|^2\int_{\Omega}|w|^2\Big)\\
        &=\lambda_1^D\int_{\Omega}|au+bw|^2\\
        &\leq \mu_2\int_{\Omega}|au+bw|^2.
    \end{align}
    Since $w$ does not vanish in $\overline{P}$ but $u$ does, the span of $\{u,w\}$ is two-dimensional. By the variational characterization of the eigenvalue problem, there then must exist $a,b\in\Cbb$ such that $au+bw$ is a second Neumann eigenfunction for $\Omega$. Since neither $w$ nor $u$ vanish in $P$ while a second Neumann eigenfunction must have a nodal arc, neither $a$ nor $b$ equals $0$. Let $e\subseteq D$ be a horizontal line segment. Then on $e$, $\partial_{\nu}w$ is constant. Since $\partial_{\nu}(au+bw)=0$ on $e$, we have that $\partial_{\nu}u$ is constant on $e$. By Lemma \ref{nonconstant}, $(P,D)$ must be a bad pair.\\
    \indent Now suppose that $(P,D)$ is a bad pair but does not consist of a rectangle with $D$ equal to two opposite shorter edges of $P$. If $P$ is a rectangle, then the statement follows from direct computation. Otherwise, $N$ contains a horizontal edge, contradicting the hypothesis that the edges in $N$ are vertical.
\end{proof}

\section{Proof of Theorem \ref{mainthm}}
The results of Sections \ref{technical} and \ref{evalineq} allow us to give short proofs of Theorem \ref{mainthm} and Corollaries \ref{simple} and \ref{nodalline}. We begin by showing a result on the lack of regularity of each second Neumann eigenfunctions $u$ on certain non-convex polygons.

\begin{prop}\label{notinH2}
    Let $P$ be a simply connected, non-convex polygon whose edges are all either vertical or horizontal line segments and which has no vertices of angle $\pi$. Let $u$ be a second Neumann eigenfunction of $P$. Then $u\in H^1(P)\setminus H^2(P)$. In particular, $c_1(u,v)\neq 0$ in expansion (\ref{neumannexpansion}) for $u$ at some non-convex vertex $v$. 
\end{prop}
\begin{proof}
    Laplace eigenfunctions are $H^1(P)$ functions by definition. Suppose that $u\in H^2(P)$. By Theorem 1 of \cite{remarksoncritset} and the fact that $P$ is non-convex, $\partial_xu$ is not identically equal to zero. Suppose toward a contradiction that $c_1(u,v)=0$ at each non-convex vertex $v$ (so that $u\in H^2(P)$ by Lemma \ref{regularity}). Then $\partial_xu\in H^1(P)$ is a mixed eigenfunction satisfying Neumann conditions on the horizontal edges of $P$ and satisfying Dirichlet conditions on the vertical edges of $P$ by Lemma \ref{derivativeBCs}. This contradicts Theorem \ref{strictineq}.
\end{proof}

\begin{proof}[Proof of Theorem \ref{mainthm}]
    Suppose that $L$ is canonically embedded. We show that $u$ has no interior critical points by proving that $\partial_xu$ and $\partial_yu$ do not vanish in $L$. Suppose that $\partial_xu$ vanishes in $L$. By Lemmas \ref{nicenodalsets} and \ref{noloops}, $\Zcal(\partial_xu)$ contains an arc $\gamma$ intersecting $L$ with distinct endpoints in $\partial L$. By  Proposition \ref{notinH2} and Lemma \ref{neumannarc}, neither of these endpoints equals the non-convex vertex of $L$. Therefore, $\partial_xu$ has a nodal domain $\Omega$ whose closure does not contain the non-convex vertex.\footnote{We do not know of a way to prove this theorem without the construction of such a domain. This is the primary reason that our method of proof does not immediately extend to other polygons whose edges are all vertical or horizontal.} By Lemma \ref{regularity}, $\phi:=(\partial_xu)\chi_{\Omega}$ is an element of $H^1(\Omega)$, where $\chi_{\Omega}$ is the indicator function for $\Omega$. Moreover, $\partial_xu$ satisfies either Neumann or Dirichlet conditions at each point on the boundary of $L$. Since $\phi$ is an eigenfunction with eigenvalue $\mu_2$, we get a contradiction to Lemma \ref{noloops}. The reasoning that $\partial_yu$ does not vanish in $L$ is similar.\\
    \indent Thus, if $u$ has a critical point, then it is contained in an edge of $L$. Suppose without loss of generality that this edge is horizontal. By Lemma \ref{cpinedge}, $\Zcal(\partial_xu)$ contains an arc intersecting $L$, yielding a contradiction as in the first paragraph.\\
    \indent It remains to show that the diametric vertices are the unique local (and hence global) extrema of $u$. Since $\partial_xu$ and $\partial_yu$ do not vanish and since there are no critical points in any edge, it suffices to show that $(\partial_xu)(\partial_yu)<0$ in $L$. Suppose to the contrary that $(\partial_xu)(\partial_yu)>0$ in $L$. Then the non-convex vertex is a local extremum of $u|_{\partial L}$, and it follows that $c_1=0$ in expansion (\ref{neumannexpansion}), contradicting Proposition \ref{notinH2}.
\end{proof}
\begin{proof}[Proof of Corollary \ref{simple}]
    Let $v$ be a diametric vertex of $L$. The map $u\mapsto u(v)$ is a linear functional on the second Neumann eigenspace. By Theorem \ref{mainthm}, the kernel of this functional is equal to $\{0\}$, so the eigenspace must be one dimensional. 
\end{proof}
\begin{proof}[Proof of Corollary \ref{nodalline}]
    By Courant's nodal domains theorem, Lemma \ref{noloops}, and the fact that $u$ is orthogonal to constant functions, the nodal set of $u$ is a simple arc dividing $L$ into two nodal domains and having distinct endpoints in $\partial L$. Since the diametric vertices $v_1$ and $v_2$ are the unique local extrema of $u$, we must have $u(v_1)u(v_2)<0$, and the first statement follows.\\
    \indent If $a_1=a_3$ and $a_2=a_4$, then $L$ is symmetric about the line $y=x$. By Corollary \ref{simple} and the fact that the Laplacian commutes with pullback operators by isometries, $u$ is either even or odd about this line. Since $(\partial_xu)(\partial_yu)<0$, $u$ cannot be even. Thus, the nodal line is the intersection of $L$ with the line $y=x$.
\end{proof}

\section{Proof of Theorem \ref{swissthm}}
Let $u$ denote an eigenfunction of the Laplacian on a Swiss cross surface $S$ tiled by an L-shaped domain $L$. By the construction of $S$, there is an isometry group generated by two orthogonal reflection symmetries $\sigma$ and $\tau$ of $S$. Since the Laplacian commutes with pullbacks by isometries, we obtain an $L^2(S)$-orthogonal decomposition of each eigenspace into functions that are either even or odd with respect to $\sigma$ or $\tau$. For example, $$u_{eo}=\frac{1}{4}\big(u+\sigma^*u-\tau^*u-\tau^*\sigma^*u\big)$$ is even with respect to $\sigma$ and odd with respect to $\tau$. By considering the eigenfunction on $S$ as an eigenfunction on the polygon $C$ (see Definition \ref{swissdef}) with periodic boundary conditions, we see then that $u_{eo}$ (and $u_{oe}$) restricts to an eigenfunction on $L$ satisfying alternating Dirichlet-Neumann boundary conditions.\footnote{By \textit{alternating Dirichlet-Neumann boundary conditions}, we mean that $u_{eo}|_L$ satisfies Neumann conditions on $3$ pairwise non-adjacent edges of $L$ and Dirichlet conditions on the other $3$ edges.} Similarly, $u_{ee}$ (resp. $u_{oo}$) restricts to an eigenfunction of $L$ satisfying Neumann (resp. Dirichlet) boundary conditions.

\begin{proof}[Proof of Theorem \ref{swissthm}]
    By Theorem \ref{mainthm} and the discussion above, it suffices to show that the second Neumann eigenvalue of $L$ is less than the first Dirichlet eigenvalue and the first mixed eigenvalues corresponding to the two sets of alternating mixed boundary conditions. Each of these inequalities follows from Theorem \ref{strictineq}.
\end{proof}

\section{More eigenvalue inequalities}
In this section we prove that there exist certain L-shaped domains for which certain eigenvalue inequalities hold. The existence of these L-shaped domains will be useful in Sections \ref{mixedthms} and \ref{Ltiledthms}. Throughout the section, we suppose that every L-shaped domain is canonically embedded. We label the edges of an L-shaped domain as in Notation \ref{edgelabels}.

\begin{lem}\label{oppositeDirineqhappens}
    If $L$ is an L-shaped domain for which $a_1=a_3$ and $a_2=a_4$, then $\lambda_1^{e_1}<\lambda_1^{e_3\cup e_5}$ and $\lambda_1^{e_2}<\lambda_1^{e_4\cup e_6}$.
\end{lem}
\begin{proof}
    Let $u$ be a first mixed eigenfunction for the problem with $D=e_3\cup e_5$ (resp. $D=e_4\cup e_6)$. Extend $u$ by reflection to a first mixed eigenfunction for the T-shaped domain $T$ (resp. U-shaped domain $U$) obtained by reflecting $L$ over $e_1$ (resp. $e_2)$. By Theorem \ref{strictineq}, the second Neumann eigenfunction of $T$ (resp. $U$) is strictly less than $\lambda_1^{e_3\cup e_5}$ (resp. $\lambda_1^{e_4\cup e_6}$). If we can show that the second Neumann eigenvalue of $T$ (resp. $U$) equals $\lambda_1^{e_1}(L)$ (resp. $\lambda_1^{e_2}(L)$), then we are done.\\
    \indent As in the proof of Theorem \ref{swissthm}, we can decompose each second Neumann eigenfunction of $T$ (resp. $U$) as a sum of eigenfunctions that are even or odd with respect to reflection about $e_1$ (resp. $e_2$). These eigenfunctions restrict, respectively, to Neumann and mixed eigenfunctions of $L$ with $D=e_1$ (resp. $D=e_2$). If the second Neumann eigenfunction of $T$ (resp. $U$) is a reflection of a second Neumann eigenfunction of $L$, then this eigenfunction has three nodal domains by the second statement of Corollary \ref{nodalline}. This contradicts Courant's nodal domains theorem. Thus each second Neumann eigenfunction is an extension of a first mixed eigenfunction on $L$ with $D=e_1$ (resp. $D=e_2$), so the second Neumann eigenvalue of $T$ equals $\lambda_1^{e_1}(L)$ (resp. $\lambda_2^{e_2}(L)$).
\end{proof}

\begin{remk}
    Motivated by numerical experimentation, we conjecture that the inequalities shown in Lemma \ref{oppositeDirineqhappens} actually hold for all L-shaped domains. However, we do not know how to prove that they hold in full generality. 
\end{remk}

\begin{lem}\label{atleastone}
    The following inequalities hold for every L-shaped domain: $$\min\{\lambda_1^{e_1},\lambda_1^{e_6}\}<\mu_2,$$ $$\min\{\lambda_1^{e_2},\lambda_1^{e_5}\}<\mu_2,$$ $$\min\{\lambda_1^{e_1},\lambda_1^{e_2}\}<\mu_2.$$
\end{lem}
\begin{proof}
    Let $u$ be a second Neumann eigenfunction of $L$. By Corollary \ref{nodalline}, the interior of at least one of $e_1$ or $e_6$ (resp. $e_2$ or $e_5$) does not contain an endpoint of an arc in $\Zcal(u)$. Suppose without loss generality that this edge equals $e_1$ (resp. $e_2$). By applying Lemma \ref{noloops} to the nodal domain of $u$ whose closure does not contain $e_1$ (resp. $e_2$), we obtain the first two inequalities. \\
    \indent Suppose that the third inequality is false, and let $u$ be a second Neumann eigenfunction for $L$. By the same argument as in the last paragraph, this implies that $\Zcal(u)$ has an endpoint in each of $e_1$ and $e_2$. By Corollary \ref{nodalline}, $\Zcal(u)$ does not have an endpoint at the vertex at which $e_1$ and $e_2$ meet. Identify $u$ with its extension via reflection to the H-shaped domain $H$ obtained by reflection over $e_1$ and $e_2$. Since the inequality is false, $u$ is a second Neumann eigenfunction of $H$. Since $\Zcal(u)$ has endpoints in both $e_1$ and $e_2$, then its extension contains a loop, contradicting Lemma \ref{noloops}.
\end{proof}
Applying Lemma \ref{atleastone} to, for example, an L-shaped domain with parameters $a_1=a_3$ and $a_2=a_4$, we see that there exist L-shaped domains where $\lambda_1^{e_1},\lambda_1^{e_2},$ $\lambda_1^{e_5}$, and $\lambda_1^{e_6}$ are all strictly less than $\mu_2$. The next lemmas show that this is not always true. 

\begin{lem}\label{longDirsometimes}
    There exists an L-shaped domain $L$ for which $\mu_2<\lambda_1^{e_1}$.
\end{lem}
\begin{proof}
    We construct this domain by using domain monotonicity to find crude estimates on eigenvalues. Let $L$ be a canonically embedded L-shaped domain with defining parameters $a_1$, $a_2$, $a_3$, and $a_4$. Let $R_1$ be a rectangle with side lengths $a_1$ and $a_3+a_4$ and $R_2$ be a rectangle with side lengths $a_2$ and $a_3$. Let $\lambda_1(R_1)$ (resp. $\lambda_1(R_2)$) be the first mixed eigenvalue for $R_1$ (resp. $R_2$) with $D$ equal to an edge of length $a_1$ (resp. $a_2$). By domain monotonicity, the first mixed eigenvalue of $L$ with Dirichlet conditions on $e_1$ satisfies the inequality
    $$\Big(\frac{\pi}{2(a_3+a_4)}\Big)^2=\min\{\lambda_1(R_1),\lambda_1(R_2)\}\leq \lambda_1^{e_1}.$$
    Let $\lambda_2(R_2)$ be the second mixed eigenvalue for $R_2$ with $D$ equal to an edge of length $a_3$. Then Dirichlet domain monotonicity gives $$\mu_2<\lambda_2(R_3)\leq\Big(\frac{3\pi}{2a_2}\Big)^2.$$ We then have $\mu_2<\lambda_1^{e_1}$ if $a_3$ and $a_4$ are sufficiently small and $a_2$ is sufficiently large. 
\end{proof}

\begin{lem}\label{perturbedevals}
    Between any two L-shaped domains $L_0$ and $L_1$, there exists a path of L-shaped domains $L_t$ for $t\in [0,1]$ and diffeomorphisms $G_t:L_0\to L_t$ such that the second Neumann and first mixed eigenvalues of $L_t$ vary analytically in $t$ for any mixed problem such that the Dirichlet region $D_t$ for $L_t$ is a fixed union of edges $e_{i_1}(L_t),...,e_{i_k}(L_t)$. For each of these problems and for each $t$, we may also find an $L^2(L_t)$-normalized eigenfunction $u_t$ such that $t\mapsto G_t^*u_t$ is real analytic as a map into $H^1(L_0)$.
\end{lem}
\begin{proof}
    We begin by showing that $L_0$ and $L_1$ can be connected by a ``linear" path of L-shaped domains. Suppose that $L_0$ (resp. $L_1$) is defined by parameters $a_1^0,a_2^0,a_3^0,a_4^0$ (resp. $a_1^1,a_2^1,a_3^1,a_4^1$). For $t\in[0,1]$ and $i\in\{1,2,3,4\}$, let $a_i^t=(1-t)a_i^0+ta_i^1$. Let $L_t$ be the L-shaped domain defined by parameters $a_1^t,a_2^t,a_3^t,a_4^t$. Let $G_t:L_0\to L_t$ be the piecewise linear homeomorphism given by $$G_t(x,y)=\begin{cases}
        \Big(\displaystyle\frac{a_1^t}{a_1^0}x,\displaystyle\frac{a_3^t}{a_3^0}y\Big)\;\;&\text{if}\;\;0<x<a_1,0<y<a_3\\\\
        \Big(\displaystyle\frac{a_1^t}{a_1^0}x,\displaystyle\frac{a_4^t}{a_4^0}(y-a_3^0)+a_3^t\Big)\;\;&\text{if}\;\;0<x<a_1,a_3\leq y<a_3+a_4\\\\
        \Big(\displaystyle\frac{a_2^t}{a_2^0}(x-a_1^0)+a_1^t,\displaystyle\frac{a_2^t}{a_2^0}y\Big)\;\;&\text{if}\;\;a_1\leq x<a_1+a_2,0<y<a_3
    \end{cases}$$
    The family of maps $\{G_t\}_{t\in[0,1]}$ forms an analytic family of uniformly bi-Lipschitz homeomorphisms. It follows then that the pullbacks under $G_t$ of the quadratic forms $H^1(L_t)\ni u\mapsto \int_{L_t}|\nabla u|^2$ and $u\mapsto \int_{L_t}|u|^2$ form real-analytic families in $t$. By standard analytic perturbation theory (see, e.g., Kato's book \cite{kato} section VII.6.4-5) and the simplicity of the first mixed and second Neumann eigenvalues (see Corollary \ref{simple}), the maps $t\mapsto \mu_2(L_t)$ and $t\mapsto \lambda_1^{D}(L_t)$ are real-analytic. We may also choose eigenfunctions $u_t$ (that are never identically equal to zero) for which the map $t\mapsto G_t^*u_t$ is analytic as a map $[0,1]\to H^1(L_0)$.\\
\end{proof}

\begin{coro}\label{longDirgeneric}
    For a generic L-shaped domain, $\lambda_1^{e_1}$, $\lambda_1^{e_6}$, and $\mu_2$ are pairwise non-equal. However, there do exist L-shaped domains for which $\lambda_1^{e_1}=\mu_2$.
\end{coro}
\begin{proof}
    Lemmas \ref{atleastone} and \ref{longDirsometimes} show that there exists an L-shaped domain $L_0$ on which these eigenvalues are pairwise non-equal. Identifying the space of L-shaped domains with the positive cone in $\Rbb^4$, every line segment in this space with an endpoint at $L_0$ is a path of L-shaped domains as in the proof of Lemma \ref{perturbedevals}. By the analyticity of the eigenvalues, the set of points on each of these paths for which the eigenvalues are not pairwise non-equal forms a set of Lebesgue measure zero. Integrating in polar coordinates and using Fubini's theorem, we see that, near $L_0$, the result holds. The result holds for the entire space of L-shaped domains by covering the space with countably many open balls that are each centered at an L-shaped domain for which these eigenvalues are pairwise non-equal. \\
    \indent The second statement follows from Lemmas \ref{atleastone}, \ref{longDirsometimes}, \ref{perturbedevals}, and the intermediate value theorem. 
\end{proof}

Proving an analogue of Lemma \ref{longDirsometimes} for a short outer edge is not as simple because bounding the eigenvalue $\lambda_1^{e_2}$ from below does not lend itself to a domain monotonicity argument. We can still find an elementary upper bound on the second Neumann eigenvalue as in the proof of Lemma \ref{longDirsometimes}. Instead of using domain monotonicity to bound $\lambda_1^{e_2}$ from below, we find a particular family of L-shaped domains for which this eigenvalue is uniformly bounded from below while the second Neumann eigenvalue can be arbitrarily small. Let $\Lcal$ be the set of L-shaped domains whose defining parameters satisfy the following conditions: $$a_1=a_2=\frac{\pi}{4},\;\;a_3\geq1,\;\;\text{and}\;\;a_4=1.$$ 

\begin{lem}\label{absbound}
    There exists a constant $C>0$ such that for all $L\in \Lcal$, $\lambda_1^{e_2}(L)\geq C$.
\end{lem}
\begin{proof}
    Let $\tilde{\Omega}$ be the L-shaped domain with parameters $a_1=a_2=\frac{\pi}{4}$, $a_3=\frac{1}{2}$, and $a_4=1$. Let $$\Omega=\tilde{\Omega}\cup \Big\{(x,y):\frac{3\pi}{8}<x<\frac{\pi}{2},\frac{1}{2}\leq y<1\Big\}.$$ Because $\Omega$ has Lipschitz boundary, Calder\'on's extension theorem (see, e.g., Theorem 11.12 in Agmon's paper \cite{agmon}) applies to provide a continuous extension operator $E:H^1(\Omega)\to H^1(\Rbb^2)$. Let $$C:=\frac{1}{2(1+\|E\|^2)}.$$
    \indent Given $L\in \Lcal$, let $R_L$ be the rectangle with vertices at $(0,0)$, $(\frac{\pi}{2},0)$, $(\frac{\pi}{2},a_3+1)$, and $(0,a_3+1)$. Let $$L'=L\cup \Big\{(x,y):\frac{3\pi}{8}<x<\frac{\pi}{2},a_3\leq y<a_3+1\Big\}.$$
    \indent On each $L\in \Lcal$, let $\lambda_L$ be the first mixed eigenvalue with Dirichlet conditions on $e_2$, and let $u_L$ be a corresponding first mixed eigenfunction with $\|u_L\|_{L^2(L)}=1$. We may extend $u_L$ by reflection to an element $u_L'\in H^1(L')$, so $\int_{L'}|\nabla u_L'|^2\leq 2\int_L|\nabla u_L|^2=2\lambda_L$. Since $(0,a_3-\frac{1}{2})+\Omega\subseteq L'$, we may use the operator $E$ to extend $u_L'$ to a function $u_L''$ on $R_L$ satisfying the inequality $$\int_{R_L}|\nabla u_L''|^2\leq \big(1+\|E\|^2\big)\int_{L'}|\nabla u_L'|^2\leq 2(1+\|E\|^2)\lambda_L.$$
    Note that, by construction, $u_L''$ vanishes on the rightmost vertical edge of $R_L$. Direct computation shows that the first eigenvalue for $R_L$ with Dirichlet conditions on the rightmost vertical edge equals $1$. The variational formulation of the eigenvalue problem then gives $$1\leq \frac{\int_{R_L}|\nabla u_L''|^2}{\int_{R_L}|u_L''|^2}\leq \int_{R_L}|\nabla u_L''|^2\leq 2(1+\|E\|^2)\lambda_L,$$ so we get $\lambda_L\geq C$.
\end{proof}

\begin{lem}\label{shortDirsometimes}
    There exists an L-shaped domain $L$ for which $\mu_2<\lambda_1^{e_2}$.
\end{lem}
\begin{proof}
    Using a rectangle of side lengths $a_1+a_2$ and $a_3$ and domain monotonicity as in the proof of Lemma \ref{longDirsometimes}, we get the inequality
    $$\mu_2\leq\Big(\frac{3\pi}{2a_3}\Big)^2,$$ which is valid for every L-shaped domain. In particular, this inequality shows that the infimum of second Neumann eigenvalues of L-shaped domains in $\Lcal$ equals $0$. This fact along with Lemma \ref{absbound} proves the lemma.
    \indent 
\end{proof}

\begin{coro}\label{shortDirgeneric}
    For a generic L-shaped domain, $\lambda_1^{e_2}$, $\lambda_1^{e_5}$, and $\mu_2$ are pairwise non-equal. However, there do exist L-shaped domains for which $\lambda_1^{e_2}=\mu_2$.
\end{coro}
\begin{proof}
    As in the proof of Corollary \ref{longDirgeneric}, this follows from Lemmas \ref{atleastone}, \ref{shortDirsometimes}, and \ref{perturbedevals}. 
\end{proof}

\section{First eigenfunctions with mixed boundary conditions}\label{mixedthms}
In this section we prove Theorems \ref{rectangleconj}, \ref{mixedthm}, \ref{longdirichlet}, and \ref{shortouterdirichlet}.

\begin{proof}[Proof of Theorem \ref{rectangleconj}]
    Let $R$ be a rectangular domain in $\Rbb^2$ with one edge contained in the $x$-axis and one edge contained in the $y$-axis. Suppose without loss of generality that $D$ is a line segment contained in the $x$-axis, and let $u$ be a corresponding first mixed eigenfunction. By Lemmas \ref{nodircp} and \ref{flatmixedarc}, there is a neighborhood $U$ of $\overline{D}$ such that $\Zcal(\partial_yu)\cap U=N\cap U$. By Lemma \ref{cpinedge}, if either vertical edge contains a critical point of $u$, then $\Zcal(\partial_yu)$ contains an arc intersecting $R$. The same is true if $R$ contains a critical point $u$. Since this arc cannot intersect $U$, $\partial_yu$ has a nodal domain $\Omega$ whose closure does not intersect $\overline{D}$. By Lemma \ref{derivativeBCs}, the restriction of $\partial_yu$ to $\Omega$ then satisfies the hypotheses of Lemma \ref{noloops}, contradicting the fact that $u$ and $\partial_yu$ have the same eigenvalue.\\
    \indent Each critical point of $u$ is thus contained in the intersection of $N$ with the edge containing $D$ or in the edge opposite to $D$, and it suffices to rule out the former possibility. The intersection of $N$ with the edge containing $D$ is either empty, a line segment, or the union of two line segments. In the first case, we are done (the result in this case also follows from direct computation of the eigenfunctions).\\
    \indent Suppose that the intersection of $N$ with the edge containing $D$ is a line segment. In this case, there is only one point on $\partial R$ for which second derivatives of $u$ are not locally integrable: the endpoint of $D$ that is not a right-angled vertex of $R$. If $\Zcal(\partial_xu)$ contains an arc intersecting $R$, therefore, it bounds a topological disk $\Omega$ such that the restriction of $\partial_xu$ is an element of $H^1(\Omega)$, and we apply Lemmas \ref{derivativeBCs} and \ref{noloops} to get a contradiction.\footnote{This proof also shows that in this case, $u$ has no non-vertex critical points at all, and the unique extremum of $u$ occurs at the vertex farthest from $D$.} \\
    \indent Now suppose that the intersection of $N$ with the edge containing $D$ is the union of two line segments $e_1$ and $e_2$ (ordered from left to right in $\Rbb^2$). If we scale $u$ to be non-negative, then by the Dirichlet condition on $D$, we have that $(\partial_xu)|_{e_1}<0$ and $(\partial_xu)|_{e_2}>0$ near the endpoints of $D$. By continuity of $\partial_xu$ in $R$, there exists an arc in $\Zcal(\partial_xu)$ with an endpoint in $D$. By similar reasoning to the previous paragraph, if there are at least two arcs in $\Zcal(\partial_xu)$, then we contradict Lemma \ref{noloops}, so there exists exactly one arc in $\Zcal(\partial_xu)$ intersecting $R$. There exists a critical point of $u$ in $e_1$ or $e_2$ if and only if this arc has an endpoint in that edge. Suppose that this arc has an endpoint in $e_1$. By our knowledge of the sign of $\partial_xu$ near the endpoints of $D$ and the fact that $\partial_yu>0$ (again, we are assuming that $u\geq 0$), the endpoint of $e_1$ that is not an endpoint of $D$ is a local minimum of $u$. This contradicts Lemma \ref{rightangledmin}.
\end{proof}
\begin{remk}
    The last paragraph of the proof of Theorem \ref{rectangleconj} demonstrates that the hypothesis that $\phi\in H^1(U)$ in Lemma \ref{noloops} is essential. In particular, in the case where $D$ does not have an endpoint at a vertex of $R$, there exists an arc in $\Zcal(\partial_xu)$ that bounds a topological disk on which $\partial_xu$ satisfies mixed Dirichlet-Neumann boundary conditions, but this does not create a contradiction due to the lack of regularity of $\partial_xu$. 
\end{remk}

\begin{proof}[Proof of Theorem \ref{mixedthm}]
    Let $L$ be a canonically embedded L-shaped domain, and suppose that $D$ satisfies the hypotheses of the theorem. Let $u$ be a corresponding first mixed eigenfunction. By applying an isometry, we may suppose that if $D$ contains a long or short outer edge or if $D$ equals a single inner edge, then this edge is parallel to the $x$-axis. If $(L,D)$ satisfies either hypothesis of the theorem, then it is not a bad pair (see Section \ref{evalineq}). By the same reasoning as in the proof of Lemma \ref{nonconstant}, $\partial_xu$ is not identically equal to $0$. By Lemmas \ref{neumannarc}, \ref{dirichletarc}, and \ref{mixedarc}, there is therefore no arc in $\Zcal(\partial_xu)$ contained in $L$ with an endpoint at the non-convex vertex. If $D$ contains a vertical edge, then this is an inner edge, and Lemma \ref{nodircp} implies that $\partial_xu$ does not vanish in this edge. By Lemma \ref{derivativeBCs}, $\partial_xu$ vanishes on the remaining edges in $D$ and satisfies either Dirichlet or Neumann conditions elsewhere on $\partial L$. Thus, if $\partial_xu$ were to vanish inside $L$, then $\partial_xu$ would have a nodal domain whose closure does not contain the vertical Dirichlet edge or the non-convex vertex, contradicting Lemma \ref{noloops}. Thus, $\partial_xu$ does not vanish in $L$ and $u$ is strictly monotonic in $x$ inside $L$. In particular, $u$ has no interior critical points. 
\end{proof}

Though Theorem \ref{mixedthm} shows that first mixed eigenfunctions for certain mixed problems are monotonic in $x$, it provides no information on whether $\partial_xu$ is positive or negative given that $u\geq 0$. However, we can show that whichever inequality holds, it holds uniformly for all L-shaped domains. This fact will be useful for locating hot spots for the mixed problems in Theorems \ref{longdirichlet} and \ref{shortouterdirichlet}. In the following two theorems we suppose that each L-shaped domain $L$ is canonically embedded.

\begin{lem}\label{ifonethenall}
    Suppose that $D_L,N_L$ are unions of edges $e_i(L)$ that provide a mixed eigenfunction problem for which every first eigenfunction is strictly monotonic in $x$ for every L-shaped domain $L$. For each $L$, let $u_L$ be a non-negative first mixed eigenfunction. If $u_{L_0}$ is increasing (resp. decreasing) in $x$ for some $L_0$, then $u_{L_1}$ is increasing (resp. decreasing) in $x$ for each other L-shaped domain $L_1$. 
\end{lem}
\begin{proof}
    By Lemma \ref{perturbedevals}, there exists a path of L-shaped domains $L_t$ and choices of first mixed eigenfunctions $u_t$ such that $t\mapsto u_t$ is analytic and $u_t\geq 0$ for all $t$. By the assumption that $u_{L_t}$ is strictly monotonic in direction $X$ for all $t$ and the continuity of $t\mapsto G_t^*u_t$, we have either $\partial_xu<0$ or $\partial_xu>0$ everywhere in $L_t$ for all $t$.
\end{proof}

We next prove that if $D$ equals a long or short outer edge, then the unique local extremum of each first mixed eigenfunction is located at the diametric vertex farthest from $D$. Let $L$ be a canonically embedded L-shaped domain with non-convex vertex $v$. 

\begin{lem}\label{mixedregularity}
    If $D=e_1$ or $D=e_2$ and $u$ is a first mixed eigenfunction for $L$, then $c_1(u,v)\neq 0$ in expansion (\ref{neumannexpansion}). 
\end{lem}
\begin{proof}
    Suppose that $D=e_1$ (resp. $D=e_2$). If $c_1(u,v)=0$, then $u\in H^2(L)$ by Lemma \ref{regularity}. Since $(L,D)$ is not a bad pair (See Definition \ref{badpair} and the proof of Lemma \ref{nonconstant}), $\partial_xu\in H^1(L)\setminus\{0\}$ (resp. $\partial_yu\in H^1(L)\setminus\{0\}$) is a first mixed eigenfunction satisfying Dirichlet boundary conditions on a set strictly containing $D$. Since $\partial_xu$ (resp. $\partial_yu$) has the same eigenvalue as $u$, this contradicts Lemma \ref{Dinclusion}.
\end{proof}

\begin{thm}\label{longdirichlet}
    Suppose that $D=e_1$. Then each first mixed eigenfunction $u$ has no non-vertex critical points, and if $u\geq 0$, then the unique local (and hence global) maximum occurs at the diametric vertex farthest from $D$. Moreover, if $u\geq 0$ and $L$ is canonically embedded, then $\partial_xu<0$ and $\partial_yu>0$.
\end{thm}
\begin{proof}
    By the reasoning given in the proof of Theorem \ref{mixedthm}, $u$ has no critical points in $e_3$ or $e_5$ since $\partial_x u\neq 0$ in $L$. If we can prove that $\partial_yu\neq 0$ in $L$, then we can also conclude that $u$ has no critical points in the interior of a vertical edge of $L$. It therefore suffices to show that if $u\geq 0$, then $\partial_yu>0$ and $\partial_xu<0$ in $L$. By Theorem \ref{mixedthm}, either $\partial_xu<0$ or $\partial_xu>0$ in $L$. By Lemma \ref{nodircp}, $\Zcal(\partial_yu)$ does not contain an arc with an endpoint in $D$. By Lemma \ref{neumannarc}, the only arc in $\Zcal(\partial_yu)$ with an endpoint at the non-convex vertex is the edge $e_3$. Thus, if $\Zcal(\partial_yu)$ contains at least two arcs intersecting $L$, then $\partial_yu$ has a nodal domain whose closure contains neither $D$ nor the non-convex vertex, and we have a contradiction to Lemma \ref{noloops}. Thus, $\Zcal(\partial_yu)$ contains at most one arc intersecting $L$. Also by Lemma \ref{noloops}, if $\partial_yu$ has two nodal domains, then the closure of one must contain the non-convex vertex, and the closure of the other must contain $D$. \\
    \indent We claim that if $u\geq 0$, then $\partial_xu<0$ if and only if $\Zcal(\partial_yu)$ contains no arc intersecting $L$. Indeed, suppose that $\partial_xu<0$. If $\partial_yu$ vanishes in $L$, then since $\partial_yu>0$ near $D$, we have $\partial_yu<0$ near the non-convex vertex. If $\partial_xu<0$, then we have that $c_1=0$ in expansion (\ref{neumannexpansion}) for the non-convex vertex. This would contradict Lemma \ref{mixedregularity}, so $\partial_yu$ cannot vanish in $L$. Now suppose that $\partial_yu$ does not vanish in $L$. Then since $u\geq 0$, we have $\partial_yu>0$ near $D$ and thus $\partial_yu>0$ in $L$. If $\partial_xu>0$, then we have $c_1=0$ for the non-convex vertex, again contradicting Lemma \ref{mixedregularity}.\\
    \indent Suppose that $\Zcal(\partial_yu)$ contains an arc intersecting $L$. Then $\partial_xu>0$ in $L$. Let $\Omega$ be the nodal domain of $\partial_yu$ whose closure contains $D$. Applying Lemma \ref{noloops} to $(\partial_yu)\chi_{\Omega}$, we then have $\lambda_1^D>\lambda_1^{e_3\cup e_5}$ since $\partial_yu$ satisfies Dirichlet conditions on $e_3\cup e_5$ (see Lemma \ref{derivativeBCs}). By Lemma \ref{oppositeDirineqhappens}, there exists an L-shaped domain $L'$ for which this inequality fails to hold, so $\partial_xu'<0$ for each non-negative first mixed eigenfunction $u'$ for $L'$ by the last paragraph. By Lemma \ref{ifonethenall}, we have $\partial_xu<0$. Thus, the last paragraph shows that $\partial_yu$ does not vanish in $L$. Since $u\geq 0$, we have $\partial_yu>0$ near $D$ and thus in all of $L$. 
\end{proof}

The proof of Theorem \ref{shortouterdirichlet} is very similar in spirit to that of Theorem \ref{longdirichlet}.

\begin{thm}\label{shortouterdirichlet}
    Suppose that $D=e_2$. Then each first mixed eigenfunction $u$ has no non-vertex critical points, and if $u\geq 0$, then the unique local (and hence global) maximum occurs at the diametric vertex farthest from $D$. Moreover, if $u\geq 0$, then $\partial_xu<0$ and $\partial_yu>0$ in $L$.
\end{thm}
\begin{proof}
    It again suffices to show that if $u\geq 0$, then $\partial_yu>0$ and $\partial_xu<0$ in $L$. By Theorem \ref{mixedthm}, either $\partial_yu<0$ or $\partial_yu>0$ in $L$. As in the proof of Theorem \ref{longdirichlet}, $\partial_xu<0$ in a neighborhood of $D$, and $\Zcal(\partial_xu)$ contains at most one arc intersecting $L$. By Lemma \ref{noloops}, if $\Zcal(\partial_xu)$ contains an arc intersecting $L$, then $\partial_xu$ has two nodal domains, the closure of one of which contains $D$ and the closure of the other contains the non-convex vertex.\\
    \indent By the same reasoning as in the proof of Theorem \ref{longdirichlet}, if $u\geq 0$, then $\partial_yu>0$ if and only if $\Zcal(\partial_xu)$ contains no arc intersecting $L$.\\
    \indent Suppose that $\Zcal(\partial_xu)$ contains an arc intersecting $L$, so $\partial_yu<0$. Since $\partial_xu$ vanishes on $e_4$ and $e_6$ we may use the nodal domain of $\partial_xu$ whose closure contains $D$ along with Lemma \ref{noloops} to see that $\lambda_1^D>\lambda_1^{e_4\cup e_6}$. By Lemma \ref{oppositeDirineqhappens}, there exists an L-shaped domain $L'$ for which this inequality fails, so the theorem holds by the same reasoning as in the proof of Theorem \ref{longdirichlet}.
\end{proof}

\section{Second Neumann eigenfunctions of L-tiled polygons}\label{Ltiledthms}
We conclude the paper by proving Theorems \ref{Ltiled} and \ref{genericsimplicity}. We do so in the order opposite to how they were introduced. 

\begin{proof}[Proof of Theorem \ref{genericsimplicity}]
    T-shaped and U-shaped domains all contain a line of symmetry about which their second Neumann eigenfunctions decompose into a sum of even and odd eigenfunctions. The even eigenfunctions restrict to second Neumann eigenfunctions on the L-shaped domains that tile the T-shaped and U-shaped domains. Odd eigenfunctions restrict to first mixed eigenfunctions of L-shaped domains with $D$ equal to a long or short outer edge for a T-shaped or U-shaped domain, respectively. The statement of the theorem thus holds for these domains by Lemma \ref{atleastone} and Corollaries \ref{longDirgeneric} and \ref{shortDirgeneric}.\\
    \indent For H-shaped, O-shaped, and Swiss cross domains, we use similar reasoning as in the first paragraph. Each of these domains has two orthogonal lines of symmetry about which the eigenfunctions decompose to be sums of even and odd eigenfunctions. By Lemma \ref{Dinclusion}, eigenfunctions that are even or odd with respect to each symmetry restrict to second Neumann or first mixed eigenfunctions (with $D$ equal to a short or long outer edge) of the L-shape that tiles the domains. The theorem then holds by Lemmas \ref{atleastone} and \ref{oddefcns} and Corollaries \ref{longDirgeneric} and \ref{shortDirgeneric}.
\end{proof}

\begin{prop}\label{generichotspots}
    Second Neumann eigenfunctions of L-tiled domains with simple second Neumann eigenvalues have no interior critical points and are monotonic in either $x$ or $y$.
\end{prop}
\begin{proof}
    By the argument used in the proof of Theorem \ref{genericsimplicity}, an L-tiled domain with a simple second Neumann eigenvalue restricts to a second Neumann eigenfunction of $L$ or a first mixed eigenfunction with $D$ equal to a short or long outer edge. The result then follows from Theorems \ref{mainthm}, \ref{longdirichlet}, and \ref{shortouterdirichlet}.
\end{proof}

Because H-shaped domains are obtained by reflecting over two different types of edges, we require one additional lemma for this case of Theorem \ref{Ltiled}.

\begin{lem}\label{oddefcns}
    Second Neumann eigenfunctions on H-shaped domains are sums of functions that are odd about one of the two axes of symmetry.
\end{lem}
\begin{proof}
    Let $H$ be an H-shaped domain tiled by an L-shaped domain $L$. The lemma is equivalent to saying that the second Neumann eigenspace of $H$ is spanned by functions that restrict to first mixed eigenfunctions of $L$ with $D=e_1$ or $D=e_2$. This follows from Lemma \ref{atleastone}.
\end{proof}

To prove Theorem \ref{Ltiled} for O-shaped domains, we need an O-shaped analogue of Proposition \ref{notinH2}.

\begin{prop}\label{Oversion}
    Let $O$ be an O-shaped domain with a second Neumann eigenfunction $u$. Then for some non-convex vertex $v$ of $O$, we have $c_1(u,v)\neq 0$.
\end{prop}
\begin{proof}
    Restricted to $L\subseteq O$, $u$ is either a first mixed eigenfunction for $L$ with $D=e_2$ or $D=e_5$ or it is a sum of such eigenfunctions. The result is therefore a corollary of Lemma \ref{mixedregularity}.
\end{proof}

\begin{proof}[Proof of Theorem \ref{Ltiled}]
    We first prove the theorem for T-shaped domains. The proof for U-shaped domains is essentially the same and we omit it. Let $T$ be a T-shaped domain obtained by reflecting a canonically embedded L-shaped domain $L$ over edge $e_1$. By Proposition \ref{generichotspots}, it remains only to prove the theorem for the case when $T$ has a two-dimensional second Neumann eigenspace. Using the simplicity of the second Neumann and first mixed eigenvalues of L-shaped domains and the proof of Theorem \ref{genericsimplicity}, every second Neumann eigenfunction of $T$ is a linear combination two eigenfunctions, one of which restricts to a second Neumann eigenfunction of $L$ and the other of which restricts to a first mixed eigenfunction of $L$ with Dirichlet conditions on $e_1$. Let $v_1$ be the non-convex vertex of $T$ with positive $y$ coordinate, and let $v_2$ be the other non-convex vertex. Since the eigenspace is two dimensional, there exist second Neumann eigenfunctions $u_1$ and $u_2$ such that $c_1(u_1,v_1)=0$ and $c_1(u_2,v_2)=0$. By Proposition \ref{notinH2}, $c_1(u_1,v_2)\neq0$ and $c_1(u_2,v_1)\neq0$, so these functions are linearly independent and span the second Neumann eigenspace. By Lemma \ref{noloops}, we have $(\partial_xu_i)(\partial_yu_i)\neq 0$ in $T$ for $i=1,2$ (the proof is similar to that of Theorem \ref{mainthm} since there is only one point at which $u_i$ does not have locally integrable second derivatives). In particular, these eigenfunctions have no non-vertex critical points. Suppose without loss of generality that $u_1$ and $u_2$ are chosen to be strictly increasing in $y$. Since $c_1(u_1,v_2)\neq 0$ (resp. $c_1(u_2,v_1)\neq0$), we must have $\partial_xu_1>0$ (resp. $\partial_xu_2<0$). The gradients of $u_1$ and $u_2$ are therefore non-zero and linearly independent at each point in $L$. Because $u_1$ and $u_2$ span the eigenspace, it therefore follows that no second Neumann eigenfunction has an interior critical point. Using this basis, it is also straightforward to see that each function in the eigenspace is monotonic in either the $x$ or $y$ direction. \\
    \indent We now prove the theorem for Swiss cross, H-shaped,and O-shaped domains. Let $\Omega$ be such a domain. By Proposition \ref{generichotspots}, it suffices to prove the theorem for the case where the second Neumann eigenvalue has multiplicity at least two. By Lemmas \ref{atleastone} and \ref{oddefcns}, the multiplicity of $\mu_2$ is at most two.\\
    \indent Note that $\Omega$ is centrally symmetric (i.e. its isometry group contains a rotation of order $2$). Suppose that $\Omega$ is embedded in $\Rbb^2$ such that its center of rotation is located at the origin and such that each of its edges is parallel to either the $x$- or $y$-axis. By Lemma \ref{atleastone} and Lemma \ref{oddefcns}, the eigenspace is spanned by functions that are odd about one of the two axes of symmetry and even about the other. Thus, every second Neumann eigenfunction is odd with respect to the central symmetry. It follows that the $x$ and $y$ partial derivatives of such an eigenfunction are even with respect to the central symmetry.\\
    \indent Label the non-convex vertex of $\Omega$ in the $i$th quadrant $v_i$. Because the eigenspace is two dimensional, there exist second Neumann eigenfunctions $u_1$ and $u_2$ such that $c_1(u_1,v_1)=c_1(u_2,v_2)=0$. By oddness, we also have $c_1(u_1,v_3)=c_1(u_2,v_4)=0$. By Propositions \ref{notinH2} and \ref{Oversion} and oddness, we then have that each of $c_1(u_1,v_2)$, $c_1(u_1,v_4)$, $c_1(u_2,v_1)$, and $c_1(u_2,v_3)$ is nonzero. It follows that $u_1$ and $u_2$ are linearly independent and span the eigenspace. As in the case of T-shaped and U-shaped domains, it suffices to show that the gradients of $u_1$ and $u_2$ are linearly independent. We begin by showing that these functions are monotonic with respect to $x$ and $y$. We show only that $\partial_xu_1\neq 0$ in $\Omega$, and the other monotonicities are proved similarly. By Lemma \ref{neumannarc}, $v_2$ and $v_4$ are not endpoints of arcs in $\Zcal(\partial_xu_1)$ that intersect $\Omega$. Thus, if $\Zcal(\partial_xu_1)\cap \Omega\neq\emptyset$, then $\partial_xu_1$ has at least two nodal domains. By Lemma \ref{noloops}, there can be exactly two such nodal domains; the closure of one of these must contain $v_2$ and the closure of the other must contain $v_4$. It follows that $\partial_xu_1$ takes opposite signs near $v_2$ and $v_4$, contradicting that $\partial_xu_1$ is even with respect to the central symmetry. This gives the desired monotonicity. Using that $c_1(u_1,v_2)$ and $c_1(u_2,v_1)$ are non-zero, we see that either $(\partial_xu_1)(\partial_yu_1)>0$ and $(\partial_xu_2)(\partial_yu_2)<0$ or vice versa (depending on which class of domains is being discussed). From this it follows that the gradients of $u_1$ and $u_2$ are linearly independent and that each eigenfunction is monotonic in either the $x$ or $y$ direction. 
\end{proof} 
\begin{remk}
    Note that, in contrast with the case of L-shaped domains, Theorem \ref{Ltiled} proves only that these domains have no \textit{interior} critical points in the case that the second Neumann eigenspace is two-dimensional. In this case, the Neumann boundary condition implies that the gradients of every eigenfunction are linearly dependent on each non-vertex boundary point, so for each non-vertex boundary point, there exists a second Neumann eigenfunction with a critical point at that point.
\end{remk}


\end{document}